\documentclass[11pt]{amsart}

\usepackage{amsmath,amssymb,amsthm}
\usepackage{hyperref} 
\usepackage{mathrsfs} 
\usepackage{geometry}
\newtheorem{theorem}{Theorem}[section]

\theoremstyle{definition}

\newcommand{\ii}{\mathrm{i}}
\newcommand{\E}{\mathrm{e}}
\newcommand{\re}{\mathrm{Re}}

\newcommand{\res}{\mathrm{Res}}
\title[Dirichlet Generating Function for Euler's Pentagonal Coefficients]{A Dirichlet Generating Function for the Coefficients of Euler's Pentagonal Number Theorem}
\author{Friedjof Tellkamp}
\date{\today}
\begin{document}
\begin{abstract}
We establish an integral representation for the Dirichlet generating function of the coefficients of Euler’s pentagonal number theorem.
The Bromwich-type integral enables analytic continuation to the entire complex plane, filling a gap in the literature and providing a new framework for studying the sequence’s analytic structure. 
Furthermore, we derive the asymptotic behavior as the variable tends to negative infinity, and give integral representations for the Euler function $\phi(q)$ and the Dedekind eta function $\eta(\tau)$.
Moreover, we obtain an explicit formula for the Dirichlet generating function at each positive integer, expressed as a finite sum.
\end{abstract}
\maketitle
\section{Introduction}
The infinite product $\prod (1 - q^n)$, commonly denoted $(q;q)_\infty$ or $\phi(q)$ and known as the Euler function, plays a central role in the theory of integer partitions and is closely related to the pentagonal number theorem \cite{Euler1780}. Euler discovered that this product expands into a power series in which the coefficients exhibit a highly structured pattern, a consequence of significant cancellation among high-degree terms:
\begin{align*}
\phi(q) = \prod_{n=1}^{\infty} (1 - q^n) 
= \sum_{n=0}^{\infty} a_n\, q^n = 1 - q - q^2 + q^5 + q^7 - q^{12} - q^{15} + \cdots ,
\end{align*}
where the coefficients $a_n$ are given by a piecewise-defined formula involving the generalized pentagonal numbers (see OEIS sequences \href{https://oeis.org/A010815}{A010815} and \href{https://oeis.org/A001318}{A001318}, \cite{OEIS})
\begin{align*}
a_n =
\begin{cases}
(-1)^m & \text{if } n = (3m^2 \pm m)/2 \text{ for some } m \in \mathbb{Z}^+, \\
1 & \text{if } n = 0, \\
0 & \text{otherwise} .
\end{cases}
\end{align*}
However, despite its significance through Euler's pentagonal number theorem, a Dirichlet generating function (DGF), apart from its standard definition
\begin{align*}
D(s) = \sum_{n=1}^\infty \frac{a_n}{n^s}, \quad \re(s) > 0 ,
\end{align*}
has not been explicitly documented in the literature. In this note, we show that the related DGF can be expressed as an Bromwich-type line integral, denoted by $D^*(s)$, valid for all $s \in \mathbb{C}$.
\section{Integral Representation}
\subsection{Main Theorem}
Let $F(z)$ be an even and $2\pi$-periodic function defined by
\begin{align*}
F(z) = -\frac{4\sqrt{3}\cos(z)}{1+2\cos(2z)} ,
\end{align*}
which vanishes only at odd multiples of $\pi/2$. All poles of $F(z)$ lie on the real axis, which occur only at $z = k \pi/3$, with $k \equiv 1,2 \pmod{3}$. We denote the residues at these poles by $r(k)$, defined as
\begin{align*}
r(k)=
\res\left(F, z=\frac{k\pi}{3}\right) = 
\begin{cases}
1 & \text{if } k \equiv  1,2 \pmod{6} , \\
-1 & \text{if } k \equiv  4,5 \pmod{6} , \\
0 & \text{otherwise}.
\end{cases}
\end{align*}
We also define the function
\begin{align*}
u(z) = \frac{(\pi - 3z)(2\pi - 3z)}{6\pi^2} ,
\end{align*}
which, as shown in the table below, takes the values of the generalized pentagonal numbers at $z = -k\pi/3$, with $k \ge 1$ and $k \equiv 1,2 \pmod{3}$:
\begin{align*}
\begin{array}{c|cccccccccccccc}
k & 1 & 2 & 3 & 4 & 5 & 6 & 7 & 8 & 9 & 10 & 11 & 12 & 13 & 14 \\ \hline
u(-k\pi/3) & 1 & 2 & - & 5 & 7 & - & 12 & 15 & - & 22 & 26 & - & 35 & 40 
\end{array}
\end{align*}
For $s$ a non-integer, $u(z)^{-s}$ induces branch cuts along the line segment connecting the real zeros at $\pi/3$ and $2\pi/3$, and along the vertical line $\re(z)=\pi/2$.
\begin{theorem}\label{thm:main}
For $s \in \mathbb{C}$ and $-\pi/3<c<\pi/3$, $D^*(s)$ is the analytical continuation of $D(s)$, defined by:
\begin{align*}
D^*(s) = \frac{1}{2\pi \ii} \int_{c - \ii\infty}^{c + \ii\infty} F(z) \, u(z)^{-s} \, \mathrm{d}z .
\end{align*}
\end{theorem}
\begin{proof}
The line integral effectively accounts for all residues located on the left half-plane, excluding the branch cuts described above. Since the residues are distributed periodically along the real axis, their sum can readily be written as:
\begin{align*}
\sum_{\substack{k \geq 1\\ k \not\equiv 0 \text{ (mod 3)}}} \res\left(F\cdot u^{-s}, z=-\frac{k\pi}{3}\right) .
\end{align*}
As $u(z)^{-s}$ is holomorphic in the left half-plane, it may be evaluated at any pole and factored out of the residue. Thus, for $\re(z) \leq 0$,
\begin{align*}
\res(F\cdot u^{-s},z)=u^{-s}(z) \, \res(F,z) .
\end{align*}
It follows that
\begin{align*}
D^*(s)
= \sum_{\substack{k \geq 1\\ k \not\equiv 0 \text{ (mod 3)}}} \frac{r(-k)}{u(-k\pi/3)^s}
=-\frac{1}{1^s}-\frac{1}{2^s}+\frac{1}{5^s}+\frac{1}{7^s}-\frac{1}{12^s}-\frac{1}{15^s}+\cdots ,
\end{align*}
which is equivalent to the standard definition of $D(s) = \sum_{n=1}^\infty a_n\,n^{-s}$ for $\re(s) > 0$, and therefore, by the identity theorem, defines the same function for all $s \in \mathbb{C}$.
\end{proof}
\begin{theorem}
$D^*(s)$ is an entire function.
\end{theorem}
\begin{proof}
The integrand $F(z)\,u(z)^{-s}$ is entire in $s$ for each $z$ in the strip $-\pi/3 <\mathrm{Re}(z)< \pi/3$, and decays exponentially in $|\mathrm{Im}(z)|$ uniformly in $s$ on compact sets due to the rapid decay of $F(\ii z)$. Hence, the function $D^*(s)$ is holomorphic for all $s \in \mathbb{C}$ by uniform convergence and analyticity under the integral sign.
\end{proof}
\subsection{Asymptotic Behavior}
While the asymptotic behavior of $D^*(s)$, as $s \to \infty$ is immediate and evaluates to $-1$, the limit as $s \to -\infty$  requires careful analysis and is derived below.
\begin{theorem}
The limit of $D^*(s)$ as $s \to -\infty$ is given by
\begin{align*}
\lim_{s \to -\infty} D^*(s) =2\sqrt{3} \cdot 6^{-s} \, \zeta(2s) .
\end{align*}
\end{theorem}
\begin{proof}
For negative $s$, the integrand $F(z) \cdot u(z)^{-s}$ on the vertical line of integration is dominated by large values of $z$. We make the substitution $z=\ii y$ and consider the limit $y \to \infty$. Further, we introduce two new functions, denoted $\tilde{F}(y)$ and $\tilde{u}_\alpha(y)$, which show the same asymptotic behavior as $F(\ii y)$ and $u(\ii y)$, respectively. For $\alpha > 0$,
\begin{align*}
\tilde{F}(y)&=-2\sqrt{3}\,\E^{-|y|}\, \sim F(\ii y) , \\
\tilde{u}_\alpha(y)&=\frac{3(\alpha-\ii y)^2}{2\pi^2} \sim u(\ii y) .
\end{align*}
With these approximations, the integral evaluates in closed form. We obtain
\begin{align*}
D^*(s) \sim \frac{1}{2 \pi} \int_{-\infty}^\infty \tilde{F}(y) \, \tilde{u}_\alpha(y)^{-s} \, \mathrm{d}y
= -2^s \cdot 3^{1/2-s} \, \pi^{2s-1} \, \E^{-\ii \pi s-\ii\alpha} \, \Gamma(1-2s,-\ii \alpha) + \text{c.c.}
\end{align*}
With $\alpha=\pi/2$, this is asymptotically equivalent to
\begin{align*}
2^{s+1} \cdot 3^{1/2-s} \, \pi^{2s-1} \, \sin(\pi s) \, \Gamma(1-2s) .
\end{align*}
By using Riemann’s functional equation for the zeta function,
\begin{align*}
\zeta(s)= 2^s\,\pi^{s-1}\,\sin\left(\frac{\pi s}{2}\right)\,\Gamma(1-s)\,\zeta(1-s) ,
\end{align*}
we find that the asymptotic behavior of $D^*(s)$ as $s \to -\infty$ is $2\sqrt{3}\cdot 6^{-s}\,\zeta(2s)$.
\end{proof}
The choice $\alpha = \pi/2$ may seem arbitrary at first glance. However, the asymptotic expression contains a phase depending on $\alpha$, which has period $2\pi$.
In this case, a shift of one-quarter of the period ensures that the zeros of the asymptotic approximation coincide with those of $D^*(s)$ at negative integers, $s \in \mathbb{Z}^-$.
\subsection{Perron's Formula}
Moreover, we show that applying Perron's formula to the DGF of the sequence yields its partial sum, which is defined as:
\begin{align*}
S(x) := \sum_{1 \leq n < x} a_n
\end{align*}
For simplicity, we restrict $x$ to real, non-integer values greater than $1$. 
According to Perron (see \S 11.12 in \cite{Apostol}), $S(x)$ is given by
\begin{align*}
\frac{1}{2\pi \ii} \int_{\kappa - \ii\infty}^{\kappa + \ii\infty} D^*(s) \frac{x^s}{s} \, \mathrm{d}s 
= \frac{1}{2\pi \ii} \int_{c - \ii\infty}^{c + \ii\infty} F(z) \,
\frac{1}{2\pi \ii} \int_{\kappa - \ii\infty}^{\kappa + \ii\infty} u(z)^{-s} \, \frac{x^s}{s} \, \mathrm{d}s \, \mathrm{d}z .
\end{align*}
Recognizing the inner integral as an inverse Mellin transform, we find that it evaluates to the Heaviside step function $\Theta(x-|u(z)|)$, effectively acting as an indicator function in the outer integral. $S(x)$ then reduces to
\begin{align*}
\frac{1}{2\pi \ii} \oint_{\mathcal{C}} F(z) \, \mathrm{d}z ,
\end{align*}
where $\mathcal{C}$ denotes a contour enclosing the domain bounded on the right by the  line $\re(z) = c$, and on the left by the curve defined by $|u(z)| = x$.
The integral represents the accumulated residues of $F(z)$ on the negative real axis from $z_-(x)$ to $c$, where $z_-(x)$ is the negative solution of $u(z) = x$,
\begin{align*}
z_\pm(x) = \frac{\pi}{2} \pm \frac{\pi}{6}\,\sqrt{1 + 24x} .
\end{align*}
Let $k_x$ be the smallest integer such that $ z_-(x) \le k_x\pi/3$. Then
\begin{align*}
S(x)=\sum_{k_x<k<0} r(k) ,
\end{align*}
which exhibits a jump discontinuity of size $(-1)^m$ at $x$ if and only if $x = (3m^2 \pm m)/2$ for some $m \in \mathbb{Z}^+$, i.e., at the generalized pentagonal numbers.
\subsection{Connection to Ordinary Generating Function}
In general\footnote{We reuse the notations $a_n$ and $D(s)$ for general sequences or functions; the context makes clear when they refer to the specific definitions earlier.}, an ordinary generating function $A(x)=\sum_{n=0}^\infty a_n x^n$, is related to its corresponding DGF, denoted $D(s)$, via (see \S 3 in \cite{Zagier})
\begin{align*}
D(s)=\frac{1}{\Gamma(s)} \int_{0}^{\infty} \left(\sum_{n=1}^\infty a_n\, \E^{-nt}\right) t^{s-1}\,\mathrm{d}t ,
\quad \text{with}\quad x=\E^{-t} ,
\end{align*}
or, equivalently, through its inverse transform
\begin{align*}
\sum_{n=1}^\infty a_n\,x^n=\frac{1}{2\pi\ii}\int_{\sigma-\ii\infty}^{\sigma+\ii\infty} \Gamma(s)\,D(s)\,(-\log x)^{-s} \, \mathrm{d}s .
\end{align*}
This identity allows us to derive an integral representation for the Euler function  $\phi(q)$, where, in the theory of $q$-series, the variable $q$ is customarily used instead of $x$. By replacing $D(s)$ with $D^*(s)$ and interchanging the integrals, we obtain
\begin{align*}
\phi(q) &= 1 + \frac{1}{2\pi\ii} \int_{c-\ii\infty}^{c+\ii\infty} F(z) \, \frac{1}{2\pi\ii} \int_{\sigma-\ii\infty}^{\sigma+\ii\infty} \Gamma(s) (-\log{q})^{-s}  \,u(z)^{-s} \, \mathrm{d}s\,\mathrm{d}z .
\end{align*}
The inner integral is an inverse Mellin transform, which evaluates to $q^{u(z)}$.
Because the integrand grows super-exponentially along the imaginary direction $z \to \pm \ii\infty$, the original vertical contour is unsuitable. We therefore deform it to a Hankel contour $\mathcal H$ encircling the positive real axis in a positive sense, which lies entirely in a region where the integrand is analytic and decreases sufficiently rapidly.
\par
Moreover, by shifting the integrand in the complex plane, we ensure that the constant term $+1$ is properly accounted for, as the residue $\mathrm{Res}(F,z=\pi/3)=1$ is no longer excluded from the contour. To this end, we define alternative functions $F'(z)$ and $u'(z)$, shifted by $\pi/2$:
\begin{align*}
F'(z)=F\left(z+\frac{\pi}{2}\right)=\frac{4\sqrt{3}\sin(z)}{1-2\cos(2z)} , \qquad
u'(z)=u\left(z+\frac{\pi}{2}\right)=\frac{3z^2}{2\pi^2}-\frac{1}{24}
\end{align*}
Thus,
\begin{align*}
\phi(q)
=1+\frac{1}{2\pi\ii} \int_\mathcal{H} F(z) \, q^{u(z)}\,\mathrm{d}z
=\frac{1}{2\pi\ii} \int_\mathcal{H} F'(z) \, q^{u'(z)}\,\mathrm{d}z , \quad |q|<1 .
\end{align*}
With $q=\E^{2\pi\ii\tau}$ and $\mathrm{Im}(\tau)>0$, this representation further enables an integral formula for the Dedekind eta function
$\eta(\tau)=q^{1/24}\phi(q)$:
\begin{align*}
\eta(\tau)=
\frac{1}{2\pi \ii} \int_\mathcal{H} F'(z) \, \exp \! \left( \frac{3\ii z^2\tau}{\pi} \right) \, \mathrm{d}z
\end{align*}
These integral representations are equivalent to sums over the residues of the integrands, which reduce to the known series:
\begin{align*}
\phi(q) = \sum_{n=-\infty}^{\infty} (-1)^n \, q^{(3n^2-n)/2},
\qquad 
\eta(\tau) = \sum_{n=-\infty}^{\infty} (-1)^n \, \E^{3 \pi \ii \left(n+\frac{1}{6}\right)^2 \tau}
\end{align*}
\subsection{Zeros and Functional Equation} The zeros of $D^*(s)$ in the strip $0<\mathrm{Re}(s)<1$ and $\mathrm{Im}(s)>0$ were computed numerically. The following zeros, ordered by increasing imaginary part, appear to be approximately correct according to our calculations; we cannot guarantee that all zeros have been identified:
\begin{align*}
\begin{array}{lllll}
z_1 = 0.88271 + 3.91652 \,\ii, &\quad& z_2 = 0.56199 + 6.01547 \,\ii, &\quad& z_3 = 0.35935 + 7.89946\,\ii, \\
z_4 = 0.27418 + 9.67421 \,\ii, &\quad& z_5 = 0.35560 + 11.3557 \,\ii, &\quad& z_6 = 0.65285 + 12.6760 \,\ii, \\
z_7 = 0.46855 + 13.8117 \,\ii, &\quad& z_8 = 0.52475 + 15.1884 \,\ii, &\quad& z_9 = 0.15548 + 17.6277 \,\ii, \\
z_{10} = 0.33322 + 19.0763 \,\ii, &\quad& z_{11} = 0.45763 + 19.9396 \,\ii, &\quad& z_{12} = 0.25780 + 21.2613 \,\ii, \\
\end{array}
\end{align*}
The irregular distribution of the real parts of the zeros suggests that no simple functional equation of the form $D^*(s)=\chi(s)D^*(1-s)$ exists.
\section{An Explicit Expression for Positive Integers}
In the preceding section, $s$ was treated as a complex parameter. We now specialize $s$ to positive integers, setting $s =k $ with $k \in \mathbb{Z}^+$. This discrete restriction allows us to derive a finite-sum expression for $D(k)$ for each fixed value of $k$, in contrast to the integral representation provided earlier. 
\subsection{Bernoulli numbers and Glaisher's \textit{G}-Numbers}
In what follows, we require the Bernoulli numbers, denoted $B_n$. Since the Bernoulli numbers $B_n$ are evaluated only at even indices $n$, the choice of $B_1^\pm= \pm 1/2$ is immaterial in this context. One way to define these numbers is via their exponential generating function:
\begin{align*}
E_B(x) = \sum_{n=0}^\infty B_n^+ \frac{x^n}{n!} = \frac{x}{1 - \E^{-x}}
\end{align*}
Additionally, we consider a rescaled variant of the Glaisher $G$-numbers \cite{Glaisher}, defined via the exponential generating function:
\begin{align*}
E_G^*(x) = \sum_{n=0}^\infty G^*(n) \frac{x^n}{n!} = \frac{3x}{2 + 4\cos(x)}
\end{align*}
This version corresponds to the original Glaisher sequence $G(n)$ through the relation
\begin{align*}
G(n)=G^*(2n+1).
\end{align*}
The scaling arises from Glaisher’s original convention, in which the $n$-th coefficient is associated with the $2n$-th power in the expansion.
For clarity, the first few values of $B_n$ and $G^*(n)$ are listed below.
\begin{align*}
\begin{array}{c|ccccccccccccc}
n & 0 & 1 & 2 & 3 & 4 & 5 & 6 & 7 & 8 & 9 & 10 & \quad & \text{OEIS sequences} \\ \hline
\rule{0pt}{2.5ex} B_{n}  & 1 & \pm\frac{1}{2} & \frac{1}{6} & 0 & -\frac{1}{30} & 0 & \frac{1}{42} & 0 & -\frac{1}{30} & 0 & \frac{5}{66} & & \text{\href{https://oeis.org/A027641}{A027641}}(n) / \text{\href{https://oeis.org/A027642}{A027642}}(n)\\
\rule{0pt}{2.5ex} G^*(n) & 0 & \frac{1}{2} & 0 & 1 & 0 & 5 & 0 & 49 & 0 & 809 & 0 & &\text{\href{https://oeis.org/A002111}{A002111}}(n/2 - 1/2), n \ge 3  
\end{array}
\end{align*}
%
\subsection{Explicit Formula for Positive Integers}
To evaluate $D(k)$ for positive integers, we now present an explicit finite-sum formula based on the definitions above.
\begin{theorem} For $k \in \mathbb{Z}^+$, $D(k)$ is given by
\begin{gather*}
D(k)=\sum_{j=0}^k 6^{k-j} \binom{-k}{k-j}\frac{(2\pi)^j}{j!}\,g(j) ,
\quad \text{where} \quad g(j)=
\begin{cases}
g_e(j) \quad\text{if $j$ is even}, \\ 
g_o(j) \quad\text{if $j$ is odd}, \\
\end{cases} \\
g_e(j)=-\frac{1}{2}(-1)^{j/2} (2^j-2)(3^j-3) \, B_j , \quad
g_o(j)=-\frac{1}{\sqrt{3}} (2^j+2) \, G^*(j) .
\end{gather*}
\end{theorem}
\begin{proof}
We derive an exponential generating function $E(x)$ whose series expansion yields the coefficients $g(n)$:
\begin{align*}
E(x) = \sum_{n=0}^\infty g(n) \, \frac{x^n}{n!}
&= -\frac{1}{2}\left(E_B^*(6x) - 2E_B^*(3x) - 3E_B^*(2x) + 6E_B^*(x)\right)
- \frac{1}{\sqrt{3}}\left(E_G^*(2x) + 2E_G^*(x)\right) \\
&= -\frac{3x\cos(2x) + \sqrt{3}x\sin(2x)}{\sin(3x)},
\end{align*}
where $E_B^*(x)$ is an auxiliary function, eliminating the alternating sign of the even-indexed Bernoulli numbers, defined by
\begin{align*}
E_B^*(x) = \frac{1}{2} \left(E_B(\ii x) + E_B(-\ii x)\right) = \frac{x}{2} \cot\left(\frac{x}{2}\right).
\end{align*}
The identity $E(x) = \sum_{n=0}^\infty g(n) x^n/n!$ can be inverted via Cauchy's integral formula, yielding, with $\gamma$ a positively oriented contour enclosing the origin,
\begin{align*}
g(n) = \frac{n!}{2\pi \ii} \oint_\gamma \frac{E(t)}{t^{n+1}} \, \mathrm{d}t.
\end{align*}
Substituting this expression for $g(n)$ into $D(k)$, and setting $\beta = 2\pi/6 = \pi/3$, we obtain
\begin{align*}
D(k)
= 6^k \sum_{j=0}^\infty \binom{-k}{k-j} g(j) \frac{\beta^j}{j!}
= \frac{6^k}{2\pi \ii} \oint_\gamma E(t) \sum_{j=0}^\infty \binom{-k}{k-j} \frac{\beta^j}{t^{j+1}} \, \mathrm{d}t.
\end{align*}
For $t \neq 0$ and $t \neq -\beta$, the inner sum evaluates by binomial series expansion to $(\beta/t)^{k}(1+t/\beta)^{-k}/t$.
Further, by applying the substitution $t \mapsto z -2\pi/3$, we arrive at
\begin{align*}
D(k) = \frac{1}{2\pi \ii} \oint_\gamma \frac{E(t)}{t} \, \left(\frac{\pi t + 3t^2}{2\pi^2}\right)^{-k} \, \mathrm{d}t
= - \frac{1}{2\pi \ii} \oint_{\gamma^*} F(z) \, u(z)^{-k} \, \mathrm{d}z ,
\end{align*}
which exhibits the same integral as the one described in \autoref{thm:main}, but with the opposite sign and a contour $\gamma^*$ encircling $2\pi/3$, instead of the original vertical line contour.
\par
The function $F(z) \cdot u(z)$ is antisymmetric under $z \mapsto \pi-z$, so its poles and residues are symmetric about $\re(z)=\pi/2$.
Thus, the sum of residues in $\re(z)>\pi/2$ equals that in $\re(z)<\pi/2$. By the global residue theorem, the total sum of all residues vanishes. Since the only unaccounted residue in the context of \autoref{thm:main} in $\re(z)<\pi/2$ is at $z=\pi/3$, it follows that
\begin{align*}
D(k)
= -\res\left(F\cdot u^{-k}, z=\frac{\pi}{3}\right)
= -\res\left(F\cdot u^{-k}, z=\frac{2\pi}{3}\right) ,
\end{align*}
which is also equal to $D^*(k)$, thereby completing the proof of the identity.
\end{proof}
Using the obtained formula, the expressions $D(k)$ can be efficiently computed.
For concreteness, we present explicit results for three specific values:
\begin{align*}
D(1)&=6-\frac{4\pi}{\sqrt{3}}&=-1.25519745693\ldots&\\
D(2)&=-108 + 16\sqrt{3}\pi + 2\pi^2&=-1.19842171457\ldots&\\
D(3)&=2160-288 \sqrt{3} \pi -36 \pi ^2-\frac{40 \pi ^3}{3 \sqrt{3}}&=-1.11483831010\ldots&
\end{align*}
\section*{Acknowledgments}
I gratefully acknowledge the OEIS Foundation for maintaining the Online Encyclopedia of Integer Sequences. The ability to look up sequences was essential to this work, and the lack of a Dirichlet generating function for A010815 motivated this study. I also appreciate the OEIS’s open and inclusive community, which welcomes contributions from individuals with a wide range of mathematical backgrounds.
\end{document}